\documentclass[a4paper,12pt]{article}
\usepackage{amsmath,amssymb,amsthm,graphics,latexsym,amsfonts}
\usepackage{fancyhdr}
\usepackage{CJK}
\usepackage{graphics}

\title{\Large Upper bounds for the achromatic and coloring numbers of a graph\thanks {Research supported by NSFC (No. 11161046)}}
\author{ {Baoyindureng Wu \footnote{
Corresponding author : Email: wubaoyin@hotmail.com (B. Wu) }}\\
\small  College of Mathematics and System Sciences, Xinjiang
University \\ \small  Urumqi, Xinjiang 830046, P.R.China \\
Clive Elphick { \footnote{ Email:
clive.elphick@gmail.com (C. Elphick) }}\\}
\date{}

\usepackage{indentfirst}
\begin{document}
\maketitle

\newtheorem{theorem}{Theorem}[section]
\newtheorem{lemma}[theorem]{Lemma}
\newtheorem{corollary}[theorem]{Corollary}
\newtheorem{conjecture}[theorem]{Conjecture}
\theoremstyle{plain} \CJKtilde
\newcommand{\D}{\displaystyle}
\newcommand{\DF}[2]{\D\frac{#1}{#2}}

{\small \noindent{\bfseries Abstract:} Dvo\v{r}\'{a}k \emph{et al.}
introduced a variant of the Randi\'{c} index of a graph $G$, denoted
by $R'(G)$, where $R'(G)=\sum_{uv\in E(G)}\frac 1 {\max\{d(u),
d(v)\}}$, and $d(u)$ denotes the degree of a vertex $u$ in $G$.
The coloring number $col(G)$ of a graph $G$ is the smallest number
$k$ for which there exists a linear ordering of the vertices of $G$
such that each vertex is preceded by fewer than $k$ of its
neighbors. It is well-known that $\chi(G)\leq col(G)$ for any graph
$G$, where $\chi(G)$ denotes the chromatic number of $G$. In this
note, we show that for any graph $G$ without isolated vertices,
$col(G)\leq 2R'(G)$, with equality if and only if $G$ is obtained
from identifying the center of a star with a vertex of a complete
graph. This extends some known results. In addition, we present some
new spectral bounds for
the coloring and achromatic numbers of a graph. \\

\noindent {\bfseries Keywords}: Chromatic number; Coloring number; Achromatic number; Randi\'{c} index \\

\section{\large Introduction}
The Randi\'{c} index $R(G)$ of a (molecular) graph $G$ was
introduced by Milan Randi\'{c} \cite{R} in 1975 as the sum of
$1/\sqrt{d(u)d(v)}$ over all edges $uv$ of $G$, where $d(u)$ denotes
the degree of a vertex $u$ in $G$. Formally,
$$R(G)=\sum\limits_{uv\in E(G)}\frac{1}{\sqrt{d(u)d(v)}}.$$ This index
is useful in mathematical chemistry and has been extensively
studied, see \cite{LG}. For some recent results on the Randi\'{c} index,
we refer to \cite{DP, LLBL, LLPDLS, LS}.

A variation of the Randi\'{c} index of a graph $G$ is called the
Harmonic index, denoted by $H(G)$, which was defined in
\cite{Fajtlowicz} as follows: $$H(G)=\sum_{uv\in E(G)} \frac 2
{d(u)+d(v)}.$$ In 2011 Dvo\v{r}\'{a}k \emph{et al}. introduced another
variant of the Randi\'{c} index of a graph $G$, denoted by $R'(G)$, which has been further studied by Knor et al \cite{Knor}.
 Formally,
$$R'(G)=\sum_{uv\in E(G)}\frac{ 1 }{\max\{d(u), d(v)\}}.$$

It is clear from the definitions that for a graph $G$,
\begin{equation} R'(G)\leq H(G)\leq R(G). \end{equation}

\vspace{3mm} The chromatic number of $G$, denoted by $\chi(G)$, is
the smallest number of colors needed to color all vertices of $G$
such that no pair of adjacent vertices is colored the same. As
usual, $\delta(G)$ and $\Delta(G)$ denote the minimum degree and the
maximum degree of $G$, respectively. The coloring number $col(G)$ of
a graph $G$ is the least integer $k$ such that $G$ has a vertex
ordering in which each vertex is preceded by fewer than $k$ of its
neighbors. The $degeneracy$ of $G$, denoted by $deg(G)$, is defined
as $deg(G)=max\{\delta(H): H\subseteq G\}$. It is well-known (see
Page 8 in \cite{JT}) that for any graph $G$,
\begin{equation}col(G)=deg(G)+1. \end{equation}

List coloring is an extension of coloring of graphs, introduced by
Vizing \cite{V} and independently, by Erd\H{o}s et al. \cite{ERT}.
For each vertex $v$ of a graph $G$, let $L(v)$ denote a list of
colors assigned to $v$. A list coloring is a coloring $l$ of
vertices of $G$ such that $l(v)\in L(v)$ and $l(x)\neq l(y)$ for any
$xy\in E(G)$, where $v, x, y\in V(G)$. A graph $G$ is $k$-choosable
if for any list assignment $L$ to each vertex $v\in V(G)$ with
$|L(v)|\geq k$, there always exists a list coloring $l$ of $G$. The
list chromatic number $\chi_l(G)$ (or choice number) of $G$ is the
minimum $k$ for which $G$ is $k$-choosable.

It is well-known that for any graph $G$,

\begin{equation}\chi(G)\leq \chi_l(G)\leq col(G)\leq \Delta(G)+1.
\end{equation} The detail of the inequalities in (3) can be found in a
survey paper by Tuza \cite{T} on list coloring.

In 2009, Hansen and Vukicevi\'{c} \cite{HV} established the
following relation between the Randi\'{c} index and the chromatic
number of a graph.

\begin{theorem}(Hansen and Vukicevi\'{c} \cite{HV}) Let $G$ be a
simple graph with chromatic number $\chi(G)$ and Randi\'{c} index
$R(G)$. Then $\chi(G)\leq2R(G)$ and equality holds if $G$ is a
complete graph,
possibly with some additional isolated vertices.\\
\end{theorem}

Some interesting extensions of Theorem 1.1 were recently obtained.

\begin{theorem} (Deng \emph{et al} \cite{Deng} ) For a graph $G$,
$\chi(G)\leq 2H(G)$ with equality if and only if $G$ is a complete
graph possibly with some additional isolated vertices.
\end{theorem}

\begin{theorem} (Wu, Yan and Yang \cite{wu14} ) If $G$ is a graph of order $n$ without isolated vertices, then $$col(G)\leq
2R(G),$$ with equality if and only if $G\cong K_n$.
\end{theorem}

Let $n$ and $k$ be two integers such that $n\geq k\geq 1$. We denote the graph obtained from identifying the center of the star $K_{1, n-k}$ with a vertex of the complete graph $K_k$ by $K_k\bullet K_{1,n-k}$. In particular, if $k\in\{1, 2\}$, $K_k\bullet K_{1,n-k}\cong K_{1,n-1}$; if $k=n$, $K_k\bullet K_{1,n-k}\cong K_n$.
The primary aim of this note is to prove stronger versions of Theorems
1.1-1.3, noting the inequalities in (1).

\begin{theorem} For a graph $G$ of order $n$ without isolated vertices, $col(G)\leq
2R'(G)$, with equality if and only if $G\cong K_k\bullet K_{1,n-k}$ for some $k\in\{1, \ldots, n\}$.
\end{theorem}

\begin{corollary}
For a graph $G$ of order $n$ without isolated vertices, $\chi(G)\leq
2R'(G)$, with equality if and only if $G\cong K_k\bullet K_{1,n-k}$ for some $k\in\{1, \ldots, n\}$.
\end{corollary}

\begin{corollary}
For a graph $G$ of order $n$ without isolated vertices, $\chi_l(G)\leq
2R'(G)$, with equality if and only if $G\cong K_k\bullet K_{1,n-k}$ for some $k\in\{1, \ldots, n\}$.
\end{corollary}

\begin{corollary}
For a graph $G$ of order $n$ without isolated vertices, $col(G)\leq 2H(G)$, with equality if and
only if $G\cong K_n$.
\end{corollary}

The proofs of these results will be given in the next section.

\vspace{3mm} A {\em complete $k$-coloring} of a graph $G$ is a
$k$-coloring of the graph such that for each pair of different
colors there are adjacent vertices with these colors. The {\it
achromatic number} of $G$, denoted by $\psi(G)$, is the maximum
number $k$ for which the graph has a complete $k$-coloring. Clearly,
$\chi(G)\leq \psi(G)$ for a graph $G$. In general,
$col(G)$ and $\psi(G)$ are incomparable. Tang \emph{et al.} \cite{tang15} proved that for a graph $G$,
$\psi(G)\leq 2R(G)$.

\vspace{3mm} In Section 3, we prove new bounds for the coloring
and achromatic numbers of a graph in terms of its spectrum, which strengthen $col(G) \le 2R(G)$ and $\psi(G) \le 2R(G)$. In
Section 4, we provide an example and propose two related conjectures.

\section{\large The proofs}

For convenience, an edge $e$ of a graph $G$ may be viewed as a
2-element subset of $V(G)$ and if a vertex $v$ is an end vertex of
$e$, we denote the other end of $e$ by $e\setminus v$. Moreover,
$\partial_G(v)$ denotes the set of edges which are incident with $v$
in $G$.

First we need the following theorem, which will play a key role in the
proof of Theorem 1.4.

\begin{theorem} If $v$ is a vertex of
$G$ with $d(v)=\delta(G)$, then
$$R'(G)-R'(G-v)\geq 0,$$ with equality if and only if
$N_G(v)$ is an independent set of $G$ and
$d(w)<d(v_i)$ for all $w\in N(v_i)\setminus \{v\}$.
\end{theorem}

\begin{proof}
Let $k=d_G(v)$. The result is trivial for $k=0$. So, let $k>0$ and
let $N_G(v)=\{v_1, \cdots, v_k\}$ and $d_i=d_G(v_i)$ for each $i$.
Without loss of generality, we may assume that $d_1\geq \cdots \geq
d_k$. Let
$$E_1=\{e: e\in\partial_G(v_1)\setminus \{vv_1\}\ \text{such that}\
d_G(e\setminus v_1)<d_G(v_1)\ \text {and}\ e\not\subseteq N_G(v)
\},$$
$$E_1'=\{e: e\in\partial_G(v_1)\setminus \{vv_1\}\ \text{such that}\ \
d_G(e\setminus v_1)=d_G(v_1)\ \text {and}\ e\subseteq N_G(v)\},$$
and for an integer $i\geq 2$,
$$E_i=\{e: e\in\partial_G(v_i)\setminus \{vv_i\}\ \text{such that}\
d_G(e\setminus v_i)<d_G(v_i)\ \text {and}\ e\not\subseteq
N_G(v)\},$$
$$E_i'=\{e: e\in\partial_G(v_i)\setminus \{vv_i\}\ \text{such that}\ \
d_G(e\setminus v_i)=d_G(v_i)\ \text {and}\ e\subseteq N_G(v)\}
\setminus (\cup_{j=1}^{j=i-1} E_j').$$ Let $a_i=|E_i|$ and
$b_i=|E_i'|$ for any $i\geq 1 $. Since $a_i+b_i\leq d_i-1$,

$$R'(G)-R'(G-v)=\sum_{i=1}^k \frac 1 {d_i}-\sum_{i=1}^k \frac {a_i+b_i}
{d_i(d_i-1)}\geq 0,$$ with equality if and only if $a_i+b_i=d_i-1$
for all $i$, i.e., $N_G(v)$ is an independent set of $G$ and
$d(w)<d(v_i)$ for all $w\in N(v_i)\setminus \{v\}$.

\end{proof}

\begin{lemma}
If $T$ is a tree of order $n\geq 2$, then $R'(T)\geq 1$, with equality if and only if
$T=K_{1,n-1}$.
\end{lemma}
\begin{proof}
By induction on $n$. It can be easily checked that $R'(K_{1,n-1})=1$. So, the assertion of the lemma is true for $n\in\{2, 3\}$. So, we assume that $n\geq 4$.
Let $v$ be a leaf of $T$. Then $T-v$ is a tree of order $n-1$. By Theorem 2.1, $R'(T)\geq R'(T-v)\geq 1$. If $R'(T)=1$, then $R'(T-v)=1$, and by the induction hypothesis, $T-v\cong K_{1, n-2}$. Let $u$ be the center of $T-v$.
We claim that $vu\in E(T)$. If this is not, then $v$ is adjacent to a leaf of $T-v$, say $x$, in $T$. Thus $d_T(x)=2$. However, since $R'(T)=R'(T-v)$ and $u\in N_T(x)$, by Theorem 2.1 that $d_T(x)>d_T(u)=n-2\geq 2$, a contradiction. This shows that $vu\in E(T)$ and $T\cong K_{1,n-1}$.
\end{proof}

\vspace{3mm}{\noindent \bf The proof of Theorem 1.4:}

Let $v_1, v_2, \ldots, v_n$ be an ordering of all vertices of $G$
such that $v_i$ is a minimum degree vertex
of $G_i=G-\{v_{i+1}, \ldots, v_n\}$ for each $i\in \{1, \ldots,
n\}$, where $G_n=G$ and $d_G(v_n)=\delta(G)$. It is well known that $deg(G)=\max\{d_{G_i}(v_i):\ 1\leq
i\leq n\}$ (see Theorem 12 in \cite{JT}). Let $k$ be the maximum number such that
$deg(G)=d_{G_k}(v_k)$, and $n_k$ the order of $G_k$. By Theorem 2.1,

\begin{equation} 2R'(G)\geq 2R'(G_{n-1})\geq \cdots\geq
2R'(G_k).\end{equation} Moreover,

\begin{eqnarray*}
2R'(G_k)&=& \sum\limits_{uv\in
E(G_k)}\frac{2}{\max\{d_{G_k}(u),d_{G_k}(v)\}}\\
&\geq& \sum\limits_{uv\in
E(G_k)}\frac{2}{\Delta(G_k)}\\
&\geq& \frac{\Delta(G_k)+(n_k-1)\delta(G_k)}{\Delta(G_k)} \\
&\geq& \delta(G_k)+1\\
&=& col(G).
\end{eqnarray*}

It follows that \begin{equation} \text {if}\ R'(G_k)=\delta(G_k)+1,
\text { then}\ G_k\cong K_k.\end{equation}

Now assume that $col(G)=2R'(G)$. Observe
that $col(G)=\max\{col(G_i):$ $G_i$ is a component of $G\}$ and $R'(G)=\sum_{i} R'(G_i)$. Thus, by the assumption that $col(G)=2R'(G)$ and $G$ has no isolated vertices, $G$ is connected and $col(G)\geq 2$. If $col(G)=2$, then $G$ is a tree. By Corollary 2.1, $G\cong K_{1,n-1}$.
Next we assume that $col(G)\geq 3$. By (4) and (5) we have $R'(G)=\cdots=R'(G_k)$,
$G_k\cong K_k$ and thus $col(G)=k$. We show $G\cong K_k\bullet K_{1,n-k}$ by induction on $n-k$.
It is easy to check that $col(K_n)=n+1=2R'(K_n)$ and thus the result is for $n-k=0$.
If $n-k=1$, then
$G_k=G-v_n$ and $R'(G_k)=R'(G)$, by Theorem 2.1,
$N_G(v_n)$ is an independent set of $G$. Combining with $G_k\cong K_k\ (k\geq 3)$,
$d_{G_n}(v_n)=1$. Thus $G=K_k\bullet K_{1,1}$.

Next assume that $n-k\geq 2$.
We consider $G_{n-1}$. Since $col(G_{n-1})=2R'(G_{n-1})$, by the induction hypothesis
$G_{n-1}\cong K_k\bullet K_{1,n-1-k}$. Without loss of generality, let $N_{G_{n-1}}(v_{k+1})=\cdots= N_{G_{n-1}}(v_{n-1})=\{v_k\}$. So, it remains to show that $N_G(v_n)=\{v_k\}$.

\vspace{3mm}\noindent {\bf Claim 1.} $N_G(v_n)\cap \{v_{k+1}, \ldots, v_{n-1}\}=\emptyset$.

If it is not, then $\{v_{k+1}, \ldots, v_{n-1}\}\subseteq N_G(v_n)$, because $d_G(v_n)=\delta(G)$.
By Theorem 2.1, $d_G(v_{n-1})>d_G(v_k)$. But, in this case, $d_G(v_{n-1})=2<d_G(v_k)$, a contradiction.

\vspace{3mm} By Claim 1, $N_G(v_n)\subseteq \{v_1, \ldots, v_k\}$.
Since $N_G(v_n)$ is an independent set and $\{v_1, \ldots, v_k\}$ is a clique of $G$, $|N_G(v_n)\cap \{v_1, \ldots, v_k\}|=1$.
If $N_G(v_n)=\{v_i\}$, where $i\neq k$, then $d_G(v_i)=k$. By Theorem 2.1, $d_G(v_i)>d_G(v_k)$.
However, $d_G(v_k)\geq n-2\geq k$, a contradiction. This shows $N_G(v_n)=\{v_k\}$ and hence $G\cong K_k\bullet K_{1,n-k}$.

It is straightforward to check that $$2R'(K_k\bullet K_{1,n-k})=2\times(\frac {k-2} 2+1)=k
=col(G).$$

\vspace{3mm}{\noindent \bf The proofs of Corollaries 1.5 and 1.6:}

\vspace{2mm} By (3) and Theorem 1.4, we have $\chi(G)\leq \chi_l(G)\leq 2R'(G)$. If $\chi(G)=2R'(G)$
(or $\chi_l(G)=2R'(G)$), then $col(G)=2R'(G)$. By Theorem 1.4, $G\cong K_k\bullet K_{1,n-k}$. On the other hand, it is easy to check that $$\chi(K_k\bullet K_{1,n-k})=\chi_l(K_k\bullet K_{1,n-k})=k=2R'(K_k\bullet K_{1,n-k}).$$

\vspace{3mm}{\noindent \bf The proof of Corollary 1.7:}

\vspace{2mm} By (1) and Theorem 1.4, we have $col(G)\leq 2H(G)$. If $col(G)=2H(G)$, then $col(G)=2R'(G)$.
By Theorem 1.4, $G\cong K_k\bullet K_{1,n-k}$. It can be checked that $k=2H(K_k\bullet K_{1,n-k})$ if and only if $k=n$, i.e., $G\cong K_n$.

\section{\large Spectral bounds}

\subsection{Definitions}

Let $\mu = \mu_1 \ge ... \ge \mu_n$ denote the eigenvalues of the adjacency matrix of $G$ and let $\pi, \nu$ and $\gamma$ denote the numbers
(counting multiplicities) of positive, negative and zero eigenvalues respectively. Then let

\[
s^+ = \sum_{i=1}^\pi \mu_i^2 \mbox{  and   } s^- = \sum_{i=n-\nu+1}^n \mu_i^2.
\]

Note that $\sum_{i=1}^n \mu_i^2 = s^+ + s^- = tr(A^2) = 2m$.

\subsection{Bounds for $\psi(G)$ and $col(G)$}

\begin{theorem}

For a graph G, $\psi(G) \le 2m/\sqrt{s^+} \le 2m/\mu \le 2R(G)$.

\end{theorem}

\begin{proof}

Ando and Lin \cite{ando} proved a conjecture due to Wocjan and Elphick \cite{wocjan} that $1 + s^+/s^- \le \chi$ and consequently $s^+ \le 2m(\chi -
1)/\chi$. It is clear that $\psi(\psi - 1) \le 2m$. Therefore:

\[
s^+ \le \frac{2m(\chi - 1)}{\chi} \le \frac{2m(\psi -
1)}{\psi} \le \frac{4m^2}{\psi^2}.
\]

Taking square roots and re-arranging completes the first half of the
proof.

Favaron \emph{et al} \cite{favaron93} proved that $R(G) \ge m/\mu$.
Therefore $\psi(G) \le 2m/\mu \le 2R(G)$.

\end{proof}

Note that for regular graphs, $2m/\mu = 2R' = 2H = 2R = n$ whereas for almost all regular graphs $2m/\sqrt{s^+} < n$.

\begin{lemma}
For all graphs, $col(col - 1) \le 2m$.

\end{lemma}

\begin{proof}

As noted above, $s^+ \le 2m(\chi - 1)/\chi$ and $\chi(\chi - 1) \le 2m$.
Therefore:

\[
\sqrt{s^+}(\sqrt{s^+} + 1) = s^+ + \sqrt{s^+} \le \frac{2m(\chi - 1)}{\chi} +
\frac{2m}{\chi} = 2m.
\]

We can show that $deg(G) \le \mu(G)$ as follows.

\[
deg(G) = \max(\delta(H) : H \subseteq G) \le \max(\mu(H) : H
\subseteq G) \le \mu(G).
\]

Therefore $col(G) \le \mu + 1$, so:

\[
col(col - 1) \le \mu(\mu + 1) \le \sqrt{s^+}(\sqrt{s^+} + 1) \le 2m.
\]

\end{proof}

We can now prove the following theorem, using the same proof as for
Theorem 3.1.

\begin{theorem}

For a graph G, $col(G) \le 2m/\sqrt{s^+} \le 2m/\mu \le 2R(G)$.
\end{theorem}

\subsection{Bounds for $s^+$}

The proof of Lemma 3.2 uses that $s^+ + \sqrt{s^+} \le 2m$, from which it follows that:

\[
\sqrt{s^+} \le \frac{1}{2}(\sqrt{8m + 1} - 1).
\]

This strengthens Stanley's inequality \cite{stanley} that:

\[
\mu \le \frac{1}{2}(\sqrt{8m + 1} - 1).
\]

However, $\sqrt{2m - n + 1} \le (\sqrt{8m + 1} - 1)/2$, and Hong \cite{hong} proved for graphs with no isolated vertices that $\mu \le \sqrt{2m - n + 1}.$ Elphick \emph{et al} \cite{elphick}  recently conjectured that for connected graphs $s^+ \le 2m - n + 1$, or equivalently that $s^- \ge n - 1$.

\section{\large Example and Conjectures}

A {\it Grundy $k$-coloring} of $G$ is a
$k$-coloring of $G$ such that each vertex is colored by the smallest
integer which has not appeared as a color of any of its neighbors.
The {\it Grundy number} $\Gamma(G)$ is the largest integer $k$, for
which there exists a Grundy $k$-coloring for $G$. It is clear that
for any graph $G$,  \begin{equation}  \Gamma(G)\leq \psi(G) \ \text{and}\  \chi(G)\leq \Gamma(G)\leq
\Delta(G)+1. \end{equation}

Note that each pair of $col(G)$ and $\psi(G)$, or $col(G)$ and $\Gamma(G)$, or $\psi(G)$ and $\Delta(G)$ is incomparable in general.

As an example of the bounds discussed in this paper, if $G=P_4$, then
$$\chi(G)=col(G)=2$$ $$\Gamma(G)=\psi(G)=\Delta(G)+1=2R'(G)=3$$
$$2H(G) = 3.67 \ \text{and}\ 2R(G) = 3.83$$
$$\mu_1(G) = 1.618 \ \text{and}\ \mu_2 = 0.618$$
$$2m/\mu_1(G) = 3.71  \ \text{and}\  2m/\sqrt{s^+} = 3.46.$$

We believe the following conjectures to be true.
\begin{conjecture}
For any graph $G$, $\psi(G)\leq 2R'(G)$.
\end{conjecture}

In view of (6), a more tractable conjecture than the above is as follows.

\begin{conjecture}
For any graph $G$, $\Gamma(G)\leq 2R'(G)$.
\end{conjecture}

\def\refname

{\hfil References}

\begin{thebibliography}{11}

\bibitem{ando}
T. Ando and M. Lin, \emph{Proof of a conjectured lower bound on the chromatic number of a graph}, Linear Algebra and Appl. 485 (2015) 480-484.



\bibitem{Deng}
H. Deng, S. Balachandran, S.K. Ayyaswamy and Y.B. Venkatakrishnan, \emph{On
the harmonic index and the chromatic number of a graph}, Discrete
Appl. Math. 16 (2013) 2740-2744.

\bibitem{DP} T.R. Divni\'{c} and L.R. Pavlovi\'{c}, \emph{Proof of the first part of the conjecture of Aouchiche
and Hansen about the Randi\'{c}}, Discrete Appl. Math. 161 (2013)
953-960.



\bibitem{Dvorak} Z. Dvo\v{r}\'{a}k, B. Lidick\'{y} and R. \v{S}krekovski, \emph{Randi\'{c} index and the diameter
of a graph}, European J. Combin. 32 (2011) 434-342.

\bibitem{elphick}
C. Elphick, M. Farber, F. Goldberg and P. Wocjan, \emph{Conjectured bounds for the sum of squares of positive eigenvalues of a graph}, (2015),
math arXiv : 1409.2079v2.

\bibitem{ERT} P. Erd\H{o}s, A. Rubin, and H. Taylor, \emph{Choosability in
graphs}, Congr. Num. 26 (1979) 125-157.

\bibitem{favaron93}
O. Favaron, M. Maheo and J. -F. Sacle, \emph{Some eigenvalue
properties in Graphs (conjectures of Graffiti II)}, Discrete Maths.
111 (1993) 197-220.

\bibitem{Fajtlowicz} S. Fajtlowicz, \emph{On conjectures of Graffiti-II},
Conr. Numer. 60 (1987) 187-197.

\bibitem{HV} P. Hansen and D. Vukicevi\'{c}, \emph{On the Randic index
and the chromatic number}, Discrete Math. 309 (2009) 4228-4234.

\bibitem{hong}
Y. Hong, \emph{Bounds on eigenvalues of graphs}, Discrete Math., 123, (1993), 65 - 74.

\bibitem{JT}T.R. Jensen and B. Toft, \emph{Graph Coloring Problems}, A Wiely-Interscience Publication, Jhon
Wiely and Sons, Inc, New York, 1995.

\bibitem{LG}X. Li and I. Gutman, \emph{Mathematical Aspects of
Randi\'{c}-Type Molecular Structure Descriptors}, Mathematical
Chemistry Monographs No. 1, Kragujevac, 2006.



\bibitem{LLBL}J. Liu, M. Liang, B. Cheng and B. Liu, \emph{A proof for a conjecture on
the Randi\'{c} index of graphs with diameter}, Appl. Math. Lett. 24
(2011) 752-756.

\bibitem{LLPDLS}  B. Liu, L.R. Pavlovic\'{c}, T.R. Divni\'{c},
J. Liu and M.M. Stojanovi\'{c}, \emph{On the conjecture of Aouchiche and
Hansen about the Randi\'{c} index}, Discrete Math. 313 (2013) 225-235.

\bibitem{LS} X. Li and Y. Shi, \emph{On a relation between the Randi\'{c} index and the chromatic
number}, Discrete Math. 310 (2010) 2448-2451.



\bibitem{Knor}  M. Knor, B. Kuzar and R. Skrekovski, \emph{Sandwiching the (generalized) Randi\'{c}
index}, Discrete Appl. Math. 181 (2015) 160-166.


\bibitem{R}M. Randi\'{c}, \emph{On characterization
of molecular branching}, J. Amer. Chem. Soc. 97 (1975) 6609-6615.

\bibitem{stanley}
R. P. Stanley, \emph{A bound on the spectral radius of graphs with $e$ edges}, Linear Algebra Appl., 87, (1987), 267 - 269.

\bibitem{tang15}
Z. Tang, B. Wu and L. Hu, \emph{The Grundy number of a graph}, math
arxiv:1507.01080 (2015).


\bibitem{T} Z. Tuza, \emph{Graph colorings with local constraints- a
survey}, Discuss. Math. Graph Theory 17 (1997) 161-228.


\bibitem{V} V.G. Vizing, \emph{Coloring the vertices of a graph in
prescribed colors} (Russian), Diskret. Analiz. 29 (1976) 3-10.

\bibitem{wocjan}
P. Wocjan and C. Elphick, \emph{New spectral bounds on the chromatic number encompassing all eigenvalues of the adjacency matrix},
Electron. J. Combin. 20(3) (2013), P39.

\bibitem{wu14}
B. Wu, J. Yan and X. Yang, \emph{Randic index and coloring number of
a graph}, Discrete Appl. Math. 178 (2014) 163 - 165.




\end{thebibliography}
\end{document}